\UseRawInputEncoding
\documentclass[12pt]{amsart}
\usepackage{setspace, amsmath, amsthm, amssymb, amsfonts, amscd, epic, graphicx, ulem, dsfont}
\usepackage[T1]{fontenc}
\usepackage{multirow}
\usepackage{bbm}
\usepackage{enumerate}

\makeatletter \@namedef{subjclassname@2010}{
  \textup{2020} Mathematics Subject Classification}
\makeatother

\newtheorem{thm}{Theorem}[section]

\newtheorem{cor}[thm]{Corollary}

\newtheorem{pro}[thm]{Proposition}

\theoremstyle{remark}
\newtheorem*{rema}{Remark}

\theoremstyle{definition}

\newtheorem{exa}[thm]{\textbf{Example}}

\newcommand{\ran}{\text{\rm{ran}}}

\newcommand{\R}{\mathbb{R}}

\newcommand{\C}{\mathbb{C}}

\begin{document}

\title[Fuglede-Putnam theorem]{Unbounded generalizations of the Fuglede-Putnam theorem}
\author[S. Dehimi, M. H. Mortad and A. Bachir]{Souheyb Dehimi, Mohammed Hichem Mortad$^*$ and Ahmed Bachir}

\thanks{* Corresponding author.}
\date{}


\address{(The first author) Department of Mathematics, Faculty of Mathematics and Informatics,
University of Mohamed El Bachir El Ibrahimi, Bordj Bou Arréridj,
El-Anasser 34030, Algeria.}

\email{souheyb.dehimi@univ-bba.dz, sohayb20091@gmail.com}

\address{(The corresponding author) Department of
Mathematics, University of Oran 1, Ahmed Ben Bella, B.P. 1524, El
Menouar, Oran 31000, Algeria.}

\email{mhmortad@gmail.com, mortad.hichem@univ-oran1.dz.}

\address{(The third author) Department of Mathematics, College of Science, King Khalid University, Abha, Saudi
Arabia.}

\email{abishr@kku.edu.sa, bachir\_ahmed@hotmail.com}

\begin{abstract}In this paper, we prove and disprove several generalizations of unbounded versions of the Fuglede-Putnam
theorem.
\end{abstract}

\maketitle

\section{Essential background}

All operators considered here are linear but not necessarily
bounded. If an operator is bounded and everywhere defined, then it
belongs to $B(H)$ which is the algebra of all bounded linear
operators on $H$ (see \cite{Mortad-Oper-TH-BOOK-WSPC} for its
fundamental properties).

Most unbounded operators that we encounter are defined on a subspace
(called domain) of a Hilbert space. If the domain is dense, then we
say that the operator is densely defined. In such case, the adjoint
exists and is unique.

Let us recall a few basic definitions about non-necessarily bounded
operators. If $S$ and $T$ are two linear operators with domains
$D(S)$ and $D(T)$ respectively, then $T$ is said to be an extension
of $S$, written as $S\subset T$, if $D(S)\subset D(T)$ and $S$ and
$T$ coincide on $D(S)$.

An operator $T$ is called closed if its graph is closed in $H\oplus
H$. It is called closable if it has a closed extension. The smallest
closed extension of it is called its closure and it is denoted by
$\overline{T}$ (a standard result states that a densely defined $T$
is closable iff $T^*$ has a dense domain, and in which case
$\overline{T}=T^{**}$). If $T$ is closable, then
\[S\subset T\Rightarrow \overline{S}\subset
\overline{T}.\] If $T$ is densely defined, we say that $T$ is
self-adjoint when $T=T^*$; symmetric if $T\subset T^*$; normal if
$T$ is \textit{closed} and $TT^*=T^*T$.

The product $ST$ and the sum $S+T$ of two operators $S$ and $T$ are
defined in the usual fashion on the natural domains:

\[D(ST)=\{x\in D(T):~Tx\in D(S)\}\]
and
\[D(S+T)=D(S)\cap D(T).\]

In the event that $S$, $T$ and $ST$ are densely defined, then
\[T^*S^*\subset (ST)^*,\]
with the equality occurring when $S\in B(H)$. If $S+T$ is densely
defined, then
\[S^*+T^*\subset (S+T)^*\]
with the equality occurring when $S\in B(H)$.

Let $T$ be a linear operator (possibly unbounded) with domain $D(T)$
and let $B\in B(H)$. Say that $B$ commutes with $T$ if
\[BT\subset TB.\]
In other words, this means that $D(T)\subset D(TB)$ and
\[BTx=TBx,~\forall x\in D(T).\]

Let $A$ be an injective operator (not necessarily bounded) from
$D(A)$ into $H$. Then $A^{-1}: \ran(A)\rightarrow H$ is called the
inverse of $A$, with $D(A^{-1})=\ran(A)$.

If the inverse of an unbounded operator is bounded and everywhere
defined (e.g. if $A:D(A)\to H$ is closed and bijective), then $A$ is
said to be boundedly invertible. In other words, such is the case if
there is a $B\in B(H)$ such that
\[AB=I\text{ and } BA\subset I.\]
If $A$ is boundedly invertible, then it is closed.

The resolvent set of $A$, denoted by $\rho(A)$, is defined by
\[\rho(A)=\{\lambda\in\C:~\lambda I-A\text{ is bijective and }(\lambda I-A)^{-1}\in B(H)\}.\]

The complement of $\rho(A)$, denoted by $\sigma(A)$,
\[\sigma(A)=\C\setminus \rho(A)\]
is called the spectrum of $A$.

\section{Introduction}

The aim of this paper is to obtain some generalizations of the
Fuglede-Putnam theorem involving unbounded operators.

Recall that the original version of the Fuglede-Putnam theorem
reads:

\begin{thm}\label{Fug-Put UNBD A BD}(\cite{FUG}, \cite{PUT}) If $A\in B(H)$ and if $M$
and $N$ are normal (non necessarily bounded) operators, then
\[AN\subset MA\Longrightarrow AN^*\subset M^*A.\]
\end{thm}

There have been many generalizations of the Fuglede-Putnam theorem
since Fuglede's paper. However, most generalizations were devoted to
relaxing the normality assumption. Apparently, the first
generalization of the Fuglede theorem to an unbounded $A$ was
established in \cite{Nussbaum-1969}. Then the first generalization
involving unbounded operators of the Fuglede-\textit{Putnam} theorem
is:

\begin{thm}\label{Fug-Put-MORTAD-PAMS-2003} If $A$ is a closed and symmetric operator and if $N$ is an unbounded normal operator, then
\[AN\subset N^*A\Longrightarrow AN^*\subset NA\]
whenever $D(N)\subset D(A)$.
\end{thm}

In fact, the previous result was established in \cite{MHM1} under
the assumption of the self-adjointness of $A$. However, and by
scrutinizing the proof in \cite{MHM1} or
\cite{Mortad-Thesis-Edinburgh-2003}, it is seen that only the
closedness and the symmetricity of $A$ were needed. Other unbounded
generalizations may be consulted in
\cite{Mortad-Fuglede-Putnham-All-unbd} and
\cite{Bens-MORTAD-FUGLEDE-DEHIMI}, and some of the references
therein. In the end, readers may wish to consult the survey
\cite{Mortad-FUG-PUT SURVEY BOOK} exclusively devoted to the
Fuglede-Putnam theorem and its applications.

\section{Generalizations of the Fuglede-Putnam theorem}

If a densely defined operator $N$ is normal, then so is its adjoint.
However, if $N^*$ is normal, then $N^{**}$ does not have to be
normal (unless $N$ itself is closed). A simple counterexample is to
take the identity operator $I_D$ restricted to some unclosed dense
domain $D\subset H$. Then $I_D$ cannot be normal for it is not
closed. But, $(I_D)^*=I$ which is the full identity on the entire
$H$, is obviously normal. Notice in the end that if $N$ is a densely
defined closable operator, then $N^*$ is normal if and only if
$\overline{N}$ is.

The first improvement is that in the very first version by B.
Fuglede, the normality of the operator is not needed as only the
normality of its closure will do. This observation has already
appeared in \cite{Boucif-Dehimi-Mortad}, but we reproduce the proof
here.

\begin{thm}\label{fuglede-type closure normal THM}
Let $B\in B(H)$ and let $A$ be a densely defined and closable
operator such that $\overline{A}$ is normal. If $BA\subset AB$, then
\[BA^*\subset A^*B.\]
\end{thm}

\begin{proof}
Since $\overline{A}$ is normal, $\overline{A}^*=A^*$ remains normal.
Now,
\begin{align*}
BA\subset AB\Longrightarrow& B^*A^*\subset A^*B^* \text{ (by taking adjoints)}\\
\Longrightarrow &B^*\overline{A}\subset \overline{A}B^* \text{ (by using the classical Fuglede theorem)}\\
\Longrightarrow& BA^*\subset A^*B \text{ (by taking adjoints
again),}
\end{align*}
establishing the result.
\end{proof}

\begin{rema}
Notice that $BA^*\subset A^*B$ does not yield $BA\subset AB$ even in
the event of the normality of $A^*$ (see \cite{Mortad-cex-BOOK}).
\end{rema}

Let us now turn to the extension of the Fuglede-Putnam version. A
similar argument to the above one could be applied.

\begin{thm}\label{fuglede-Putnam type closure normal THM}
Let $B\in B(H)$ and let $N,M$ be densely defined closable operators
such that $\overline{N}$ and $\overline{M}$ are normal. If
$BN\subset MB$, then
\[BN^*\subset M^*B.\]
\end{thm}

\begin{proof} Since $BN\subset MB$, it ensues that $B^*M^*\subset
N^*B^*$. Taking adjoints again gives $B\overline{N}\subset
\overline{M}B$. Now, apply the Fuglede-Putnam theorem to the normal
$\overline{N}$ and $\overline{M}$ to get the desired conclusion
\[BN^*\subset M^*B.\]
\end{proof}

Jab{\l}o\'{n}ski et al. obtained in \cite{Jablonski et al 2014} the
following version.

\begin{thm}\label{jablonsky et al FUG type THM}
If $N$ is a normal (bounded) operator and if $A$ is a closed densely
defined operator with $\sigma(A)\neq \C$, then:
\[NA\subset AN\Longrightarrow g(N)A\subset Ag(N)\]
for any bounded complex Borel function $g$ on $\sigma(N)$. In
particular, we have $N^*A\subset AN^*$.
\end{thm}

\begin{rema}
It is worth noticing that B. Fuglede obtained, long ago, in
\cite{FUG-1954-cexp-proper extension} a unitary $U\in B(H)$ and a
closed and symmetric $T$ with domain $D(T)\subset H$ such that
$UT\subset TU$ but $U^*T\not\subset TU^*$.
\end{rema}

Next, we give a generalization of Theorem \ref{jablonsky et al FUG
type THM} to an unbounded $N$, and as above, only the normality of
$\overline{N}$ is needed.

\begin{thm}\label{THM JAblonsky's et al generalization UNB}
Let $p$ be a one variable complex polynomial. If $N$ is a densely
defined closable operator such that $\overline{N}$ is normal and if
$A$ is a densely defined operator with $\sigma[p(A)]\neq \C$, then
\[NA\subset AN\Longrightarrow N^*A\subset AN^*\]
whenever $D(A)\subset D(N)$.
\end{thm}

\begin{rema}
This is indeed a generalization of the bounded version of the
Fuglede theorem. Observe that when $A,N\in B(H)$, then
$\overline{N}=N$, $D(A)=D(N)=H$, and $\sigma[p(A)]$ is a compact
set.
\end{rema}

\begin{proof}First, we claim that $\sigma(A)\neq \C$, whereby $A$ is
closed. Let $\lambda$ be in $\C\setminus \sigma[p(A)]$. Then, and as
in \cite{Dehimi-Mortad-squares-polynomials}, we obtain
\[p(A)-\lambda I=(A-\mu_1 I)(A-\mu_2I)\cdots (A-\mu_nI)\]
for some complex numbers $\mu_1$, $\mu_2$, $\cdots$, $\mu_n$. By
consulting again \cite{Dehimi-Mortad-squares-polynomials}, readers
see that $\sigma(A)\neq \C$.

Now, let $\lambda\in \rho(A)$. Then
\[NA\subset AN\Longrightarrow NA-\lambda N\subset AN-\lambda N=(A-\lambda I)N.\]
Since $D(A)\subset D(N)$, it is seen that $NA-\lambda N=N(A-\lambda
I)$. So
\[ N(A-\lambda I)\subset (A-\lambda I)N\Longrightarrow (A-\lambda
I)^{-1}N\subset N(A-\lambda I)^{-1}.\]

Since $\overline{N}$ is normal, we may now apply Theorem
\ref{fuglede-type closure normal THM} to get
\[(A-\lambda I)^{-1}N^*\subset N^*(A-\lambda I)^{-1}\]
because $(A-\lambda I)^{-1}\in B(H)$. Hence
\[N^*A-\lambda N^*\subset N^*(A-\lambda I)\subset (A-\lambda I)N^*=AN^*-\lambda N^*.\]
But
\[D(AN^*)\subset D(N^*)\text{ and }D(N^*A)\subset D(A)\subset D(N)\subset D(\overline{N})=D(N^*).\]
Thus, $D(N^*A)\subset D(AN^*)$, and so
\[N^*A\subset AN^*,\]
as needed.
\end{proof}

Now, we present a few consequences of the preceding result. The
first one is given without proof.

\begin{cor}
If $N$ is a densely defined closable operator such that
$\overline{N}$ is normal and if $A$ is an unbounded self-adjoint
operator with $D(A)\subset D(N)$, then
\[NA\subset AN\Longrightarrow N^*A\subset AN^*.\]
\end{cor}

\begin{cor}
If $N$ is a densely defined closable operator such that
$\overline{N}$ is normal and if $A$ is a boundedly invertible
operator, then
\[NA\subset AN\Longrightarrow N^*A\subset AN^*.\]
\end{cor}

\begin{proof}We may write
\[NA\subset AN\Longrightarrow NAA^{-1}\subset ANA^{-1}\Longrightarrow A^{-1}N\subset NA^{-1}.\]
Since $A^{-1}\in B(H)$ and $\overline{N}$ is normal, Theorem
\ref{fuglede-type closure normal THM} gives
\[A^{-1}N^*\subset N^*A^{-1} \text{ and so }N^*A\subset AN^*,\]
as needed.

\end{proof}

A Putnam's version seems impossible to obtain unless strong
conditions are imposed. However, the following special case of a
possible Putnam's version is worth stating and proving. Besides, it
is somewhat linked to the important notion of anti-commutativity.

\begin{pro}
If $N$ is a densely defined closable operator such that
$\overline{N}$ is normal and if $A$ is a densely defined operator
with $\sigma(A)\neq \C$, then
\[NA\subset -AN\Longrightarrow N^*A\subset -AN^*\]
whenever $D(A)\subset D(N)$.
\end{pro}

\begin{proof}Consider
\[\widetilde{N}=\left(
                  \begin{array}{cc}
                    N & 0 \\
                    0 & -N \\
                  \end{array}
                \right)\text{ and }\widetilde{A}=\left(
                                                   \begin{array}{cc}
                                                     0 & A \\
                                                     A & 0 \\
                                                   \end{array}
                                                 \right)
\]
where $D(\widetilde{N})=D(N)\oplus D(N)$ and
$D(\widetilde{A})=D(A)\oplus D(A)$. Then $\overline{\widetilde{N}}$
is normal and $\widetilde{A}$ is closed. Besides
$\sigma(\widetilde{A})\neq \C$. Now
\[\widetilde{N}\widetilde{A}=\left(
                                                   \begin{array}{cc}
                                                     0 & NA \\
                                                     -NA & 0 \\
                                                   \end{array}
                                                 \right)\subset \left(
                                                   \begin{array}{cc}
                                                     0 & -AN \\
                                                     AN & 0 \\
                                                   \end{array}
                                                 \right)=\widetilde{A}\widetilde{N}\]
for $NA\subset -AN$. Since $D(\widetilde{A})\subset
D(\widetilde{N})$, Theorem \ref{THM JAblonsky's et al generalization
UNB} applies, i.e. it gives $\widetilde{N}^*\widetilde{A}\subset
\widetilde{A}\widetilde{N}^*$ which, upon examining their
           entries, yields the required result.
\end{proof}

We finish this section by giving counterexamples to some
"generalizations".

\begin{exa}(\cite{Mortad-Fuglede-Putnham-All-unbd}) Consider the unbounded linear operators $A$ and $N$ which are defined by
\[Af(x)=(1+|x|)f(x)\text{ and } Nf(x)=-i(1+|x|)f'(x)\]
(with $i^2=-1$) on the domains
\[D(A)=\{f\in L^2(\R): (1+|x|)f\in L^2(\R)\}\]
and
\[D(N)=\{f\in L^2(\R): (1+|x|)f'\in L^2(\R)\}\]
respectively, and where the derivative is taken in the
distributional sense. Then $A$ is a boundedly invertible, positive,
self-adjoint unbounded operator. As for $N$, it is an unbounded
normal operator $N$ (details may consulted in
\cite{Mortad-Fuglede-Putnham-All-unbd}). It was shown that such that
\[AN^*=NA\text{ but }AN\not\subset N^*A \text{ and }N^*A\not\subset AN\]
(in fact $ANf\neq N^*Af$ for all $f\neq 0$).

So, what this example is telling us is that $NA=AN^*$ (and not just
an "inclusion"), that $N$ and $N^*$ are both normal, $\sigma(A)\neq
\C$ (as $A$ is self-adjoint), but $NA\not\subset AN^*$.
\end{exa}

This example can further be beefed up to refute certain possible
generalizations.

\begin{exa}\label{hghgdfsdzretrutioyoyoypypypy}(Cf. \cite{Mortad-Fuglede-ultimate generalization}) There exist a closed operator $T$ and a normal $M$ such that $TM\subset MT$ but $TM^*\not\subset M^*T$ and $M^*T\not\subset
TM^*$. Indeed, consider
\[M=\left(
      \begin{array}{cc}
        N^* & 0 \\
        0 & N \\
      \end{array}
    \right)\text{ and } T=\left(
                            \begin{array}{cc}
                              0 & 0 \\
                              A & 0 \\
                            \end{array}
                          \right)
\]
where $N$ is normal with domain $D(N)$ and $A$ is closed with domain
$D(A)$ and such that $AN^*=NA$\text{ but }$AN\not\subset N^*A$
\text{ and }$N^*A\not\subset AN$ (as defined above). Clearly, $M$ is
normal and $T$ is closed. Observe that $D(M)=D(N^*)\oplus D(N)$ and
$D(T)=D(A)\oplus L^2(\R)$. Now,
\[TM=\left(
                            \begin{array}{cc}
                              0 & 0 \\
                              A & 0 \\
                            \end{array}
                          \right)\left(
      \begin{array}{cc}
        N^* & 0 \\
        0 & N \\
      \end{array}
    \right)=\left(
              \begin{array}{cc}
                0_{D(N^*)} & 0_{D(N)} \\
                AN^* & 0 \\
              \end{array}
            \right)=\left(
              \begin{array}{cc}
                0 & 0_{D(N)} \\
                AN^* & 0 \\
              \end{array}
            \right)
    \]
where e.g. $0_{D(N)}$ is the zero operator restricted to $D(N)$.
Likewise
\[MT=\left(
      \begin{array}{cc}
        N^* & 0 \\
        0 & N \\
      \end{array}
    \right)\left(
                            \begin{array}{cc}
                              0 & 0 \\
                              A & 0 \\
                            \end{array}
                          \right)=\left(
                            \begin{array}{cc}
                              0 & 0 \\
                              NA & 0 \\
                            \end{array}
                          \right).\]
Since $D(TM)=D(AN^*)\oplus D(N)\subset D(NA)\oplus L^2(\R)=D(MT)$,
it ensues that $TM\subset MT$. Now, it is seen that
\[TM^*=\left(
                            \begin{array}{cc}
                              0 & 0 \\
                              A & 0 \\
                            \end{array}
                          \right)\left(
      \begin{array}{cc}
        N & 0 \\
        0 & N^* \\
      \end{array}
    \right)=\left(
              \begin{array}{cc}
                0 & 0_{D(N^*)} \\
                AN & 0 \\
              \end{array}
            \right)\]
and
\[M^*T=\left(
      \begin{array}{cc}
        N & 0 \\
        0 & N^* \\
      \end{array}
    \right)\left(
                            \begin{array}{cc}
                              0 & 0 \\
                              A & 0 \\
                            \end{array}
                          \right)=\left(
                            \begin{array}{cc}
                              0 & 0 \\
                              N^*A & 0 \\
                            \end{array}
                          \right).\]

Since $ANf\neq N^*Af$ for any $f\neq 0$, we infer that
$TM^*\not\subset M^*T$ and $M^*T\not\subset TM^*$.
\end{exa}

\end{document}